\documentclass[11pt]{amsart}
\usepackage[pagewise]{lineno}

\usepackage{pstricks,pst-plot}
\usepackage{mathrsfs}

\usepackage{amsfonts}
\usepackage{amsthm}

\theoremstyle{plain}
\newtheorem{theorem}{Theorem}[section]
\newtheorem{lemma}[theorem]{Lemma}
\theoremstyle{definition}
\newtheorem{definition}{Definition}[section]

\newcommand{\ov}{\overline}

\begin{document}
\title{Localization for Uniform Algebras\\
Generated by Real-Analytic Functions}
\author{John T. Anderson}
\address{Department of Mathematics and Computer Science, College of the Holy Cross, Worcester, MA 01610}
\email{janderso@holycross.edu}
\author{Alexander J. Izzo}
\address{Department of Mathematics and Statistics, Bowling Green State University, Bowling Green, OH, 43403 USA}
\email{aizzo@bgsu.edu}

{\maketitle}

\begin{abstract}
It is shown that if $A$ is a uniform algebra generated by real-analytic functions on a suitable compact subset $K$ of a real-analytic variety such that the maximal ideal space of $A$ is $K$, and every continuous function on $K$ is locally a uniform limit of functions in $A$, then $A=C(K)$.  This gives an affirmative answer to a special case of a question from the Proceedings of the Symposium on Function Algebras held at Tulane University in 1965.
\end{abstract}


\section{Introduction} \label{intro}

Let $A$ be a uniform algebra on a compact space $X$, i.e., a subalgebra of the space $C(X)$ of continuous complex-valued functions on $X$, containing the constant functions, separating points of $X$, and closed in the
supremum norm
\[ \|f\|_{X} = \sup\{ |f(x)|: x \in X \}. \]
We let $\mathfrak{M}_{A}$ denote the maximal ideal space of the uniform algebra $A$, as usual identified with the set of all non-zero algebra homomorphisms of $A$ into $\mathbb{C}$.  For $x \in X$, the point evaluation functional $f \mapsto f(x)$ thus defines an element of $\mathfrak{M}_{A}$; we write $\mathfrak{M}_{A} = X$ if each element of $\mathfrak{M}_{A}$ can be identified with point evaluation at some point of $X$.  Throughout the paper, all spaces will be assumed to be Hausdorff.

There are many possible approaches to the problem of characterizing $C(X)$ among the uniform algebras on $X$. For example, a theorem of Rainwater \cite{rainwater} has the following consequence: if every continuous function on $X$ is locally in $A$, then $A = C(X)$:

\begin{theorem} Let $A$ be a uniform algebra on a compact  space $X$, and suppose for each $f \in C(X)$ and each $x \in X$ there exists a neighborhood $U$ of $x$ and a function $g \in A$ with $f = g$ on $U$.  Then $A = C(X)$.  \end{theorem}

It is an open question (posed in \cite{birtel}, p. 348) whether or not this result
continues to hold if we replace the condition that each function $f$ in $C(X)$ is locally in $A$ by the condition that each function $f$ in $C(\mathfrak{M}_{A})$ is locally uniformly approximable by functions in $A$, i.e., that for each $f \in C(\mathfrak{M}_{A})$ and each $x \in \mathfrak{M}_{A}$ there exists a
neighborhood $U$ of $x$ with $f\big|_{U} \in \overline{A|_{U}}$.  The second author \cite{izzo_localization} established this for algebras on two-manifolds generated by
continuously differentiable functions.  For two-manifolds, via the device of pushing forward measures by functions in the algebra to the complex plane, one may reduce approximation problems to questions about the algebra
$R(K)$ (see \cite{freeman}) consisting of uniform limits on a compact set $K$ in the plane of rational functions with poles off $K$.  The algebra $R(K)$ is known to be local - a function locally in $A$ belongs to $A$ (see \cite[II.10]{Ga}) - and a theorem of Alexander \cite{alexander} provides a generalization of this fact: \emph{if $K$ is the union of a countable collection $\{K_{n}\}$ of compact sets  with $R(K_{n}) = C(K_{n})$ for each $n$, then $R(K) = C(K)$.}
In \cite{izzo_localization} a more general result on $R(K)$ is established, prompting the following definition:

\begin{definition}If $A$ is a uniform algebra on a compact  space $X$, we say that $A$ has the \emph{countable approximation property} if for each $f \in C(X)$ there exists a countable collection $\{M_{n}\}$ of compact subsets of $X$ with $\cup_{n=1}^{\infty} M_{n} = X$ and $f\big|_{M_{n}} \in \overline{A\big|_{M_{n}}}$ for each $n$.   \end{definition}

Note that the sets $M_{n}$ in the preceding definition are allowed to depend on $f$.  Also, if $A$ has the property that each $f \in C(X)$ is locally uniformly approximable by functions in $A$, then $A$ clearly has the countable approximation property.  We can now state the main result of \cite{izzo_localization}:

\begin{theorem}[\cite{izzo_localization}, Theorem~4.1]  Let $M$ be a compact two-dimensional manifold, and let $A$ be a uniform algebra on $M$ generated by $C^{1}$ functions.  Assume that the maximal ideal space of $A$ is $M$, and that $A$ has the countable approximation property. Then $A = C(M)$.  \end{theorem}

Our goal in this paper is to establish a similar theorem in the context of real-analytic varieties of arbitrary dimension.  That a result along these lines should hold at least in the context of real-analytic manifolds of dimension three was suggested by Wermer (private communication).

Throughout the paper, given a real-analytic subvariety $V$ of an open set $\Omega \subset \mathbb{R}^n$, and given a subset $K$ of $V$, whenever we say
\lq\lq a collection $\Phi$ of functions real-analytic on $K$,\rq\rq\ we mean that to each member $f$ of $\Phi$ there corresponds a neighborhood of $K$ in $\mathbb{R}^n$, that may depend on the function $f$, to which $f$ extends to be real-analytic.  Also, $\partial K$ will denote the boundary of $K$ relative to $V$, and ${\rm int}(K)$ will denote the interior of $K$ relative to $V$.

The precise statement of our main result is as follows.

\begin{theorem} \label{main_theorem} Let $V$ be a real-analytic subvariety of an open set $\Omega \subset \mathbb{R}^n$, and let $K$ be a compact subset of $V$ such that $\partial K$ is a real-analytic subvariety of $V$.  Let $A$ be a uniform algebra on $K$ generated by a collection $\Phi$ of functions real-analytic on $K$.  Assume that the maximal ideal space of $A$ is $K$, and that $A$ has the countable approximation property. Then $A = C(K)$.
\end{theorem}

Since $A$ satisfies the countable approximation property if each function in 
$C(K)$ is locally uniformly approximable by functions in $A$,
the above theorem answers the question from \cite{birtel} mentioned above in the special case of uniform algebras generated by real-analytic functions.
The choice of this setting was suggested by work of the authors and Wermer
 (\cite{anderson_izzo_two_manifolds}, \cite{anderson_izzo_wermer_three_dimensions}, \cite{anderson_izzo_wermer_varieties}, \cite{anderson_izzo_smooth_manifolds},  \cite{anderson_izzo_varieties}) on the peak-point conjecture.
This conjecture concerned a possible characterization of $C(X)$ which was shown to be false in general by Cole \cite{Co} but is true in a number of special cases of interest.  For a review of the peak-point conjecture to 2010, see the survey \cite{izzo_survey} of the second author; our proofs in the present paper follow closely the proofs in \cite{anderson_izzo_varieties}, with necessary modifications.  For convenience of the reader we repeat the general construction and details of the proofs.

We also remark that if the variety $V$ is a compact holomorphically convex subset of $\mathbb{C}^n$, a theorem of Stout \cite{stout_varieties} implies, without additional hypotheses, that the algebra $\mathcal{O}(V)$ consisting of the uniform closure of functions holomorphic in a neighborhood (dependent on the function) of $V$ is equal to $C(V)$.  In particular, this implies that if $V$ is polynomially convex (respectively rationally convex), then the algebra $P(V)$ of uniform limits of polynomials on $V$ (respectively $R(V)$, the uniform limits of rational functions holomorphic in a neighborhood of $V$) satisfies $P(V) = C(V)$ (respectively $R(V) = C(V)$).

 To prove Theorem \ref{main_theorem} we will argue using duality: to show that $A = C(X)$ it suffices to prove that if $\mu$ is a measure supported on $X$ with $\mu \in A^{\perp}$, i.e., such that
\[ \int f d\mu = 0 \mbox{ for all }f \in A, \]
then $\mu = 0$.  As will be shown at the beginning of Section \ref{proof}, Theorem~\ref{main_theorem} is an easy consequence of the following theorem.

\begin{theorem} \label{support_theorem} Let $V$ be a real-analytic subvariety of an open set $\Omega \subset \mathbb{R}^n$, and let $K$ be a compact subset of $V$.  Let $A$ be a uniform algebra on $K$ generated by a collection $\Phi$ of functions real-analytic on $K$.  Assume that the maximal ideal space of $A$ is $K$ and that $A$ satisfies the countable approximation property.  Then $\mu \in A^{\perp}$ implies\/ ${\rm supp}(\mu) \cap {\rm int}(K) = \emptyset$. \end{theorem}

In Section \ref{prelims} we collect some preliminary lemmas and comment on the general outline of the proof of Theorem~\ref{support_theorem}, which is presented in Section \ref{proof}.

We conclude this introduction by commenting on the relationship of the results in this paper with the concepts of a local (or approximately local) algebra, and a regular (or approximately regular) algebra.  Given a uniform algebra $A$ with maximal ideal space $\mathfrak{M}_{A}$, we say that $f \in C(\mathfrak{M}_{A})$ \emph{belongs locally to $A$} if for each point $x \in \mathfrak{M}_{A}$ there is a neighborhood $U$ of $x$ in $\mathfrak{M}_{A}$ and a function $g \in A$ with $f = g$ on $U$, and we say that $f$ is \emph{locally approximable by $A$} if for each $x \in  \mathfrak{M}_{A}$ there is a neighborhood $U$ of $x$ in $\mathfrak{M}_{A}$ on which $f$ is a uniform limit of functions in $A$.  (Here we identify elements of $A$ with their extensions to elements of $C(\mathfrak{M}_{A})$ via the Gelfand transform.)  We say that $A$ is \emph{local} if each function belonging locally to $A$ is in $A$, and that $A$ is \emph{approximately local} if each function locally approximable by $A$ is in $A$.  Note that every approximately local algebra is local.  The first example of a nonlocal uniform algebra was given by Kallin \cite{K}; a result of Izzo \cite{izzo_localization}, \cite{izzo_nonlocal} asserts that nonlocal algebras, with $C^{\infty}$ generators, exist on any compact $C^{\infty}$ manifold of dimension at least three.

We say that a uniform algebra $A$ is \emph{regular} if for each closed subset $E$ of $\mathfrak{M}_{A}$ and point $p \in \mathfrak{M}_{A}\setminus E$, there exists $f \in A$ with $f$ vanishing on $E$ while $f(p) = 1$.  We say that $A$ is \emph{approximately regular} if for each closed subset $E$ of $\mathfrak{M}_{A}$ and point $p \in \mathfrak{M}_{A}\setminus E$, and each $\epsilon > 0$, there exists $f \in A$ with $\|f\| < \epsilon$ while $f(p) = 1$.

Connecting these concepts, a result of Wilken \cite{Wilken} states that every regular algebra is approximately local.  In the same paper Wilken proved that every approximately regular algebra is two-local, a version of locality in which one allows just two sets on which the given function must agree with an element of the algebra.  More precisely, we say that $A$ is \emph{two-local} if whenever $f \in C(\mathfrak{M}_{A})$ has the property that there exist open sets $U_{1}, U_{2} \subset\mathfrak{M}_{A}$ with $U_{1} \cup U_{2} = \mathfrak{M}_{A}$, and functions $g_{1}, g_{2} \in A$ with $f\big|_{U_{j}} = g_{j}$  ($j=1,2$), then $f \in A$.

It is easy to check that $A$ is approximately regular if and only if it has the following property: for each closed subset $E$ of $\mathfrak{M}_{A}$,
\begin{equation} \mathfrak{M}_{\ov{A|_{E}}} = E. \label{eq: max ideal restriction} \end{equation}
It is shown below (Lemma \ref{inheritance lemma}) that (\ref{eq: max ideal restriction}) holds if $A$ has the countable approximation property as an algebra on $\mathfrak{M}_{A}$, and (\ref{eq: max ideal restriction}) holds also (Lemma 2.3 of \cite{anderson_izzo_smooth_manifolds}) if every point of $\mathfrak{M}_{A}$ is a peak point for $A$.  Thus the algebras considered here and in \cite{anderson_izzo_smooth_manifolds} are two-local by Wilken's theorem.

To the best of the authors' knowledge, it is unknown whether every approximately regular uniform algebra is local.  In particular, we may ask (i) if $A$ is a uniform algebra such that each point of $\mathfrak{M}_{A}$ is a peak point for $A$ is $A$ necessarily local (or approximately local)? (ii) if $A$ is a uniform algebra on its maximal ideal space with the countable approximation property is $A$ necessarily local (or approximately local)?  By Theorem~2.3 of \cite{izzo_localization}, if $A$ is a uniform algebra on a compact metrizable space $X$ such that $A$ is both approximately local and satisfies the countable approximation property, then $A = C(X)$.  Therefore if the answer to question (ii) in the case of approximate localness is affirmative, then the countable approximation property actually characterizes $C(X)$ among the uniform algebras $A$  on a compact metrizable space $X$ with $\mathfrak{M}_{A} = X$.

\section{Preliminaries} \label{prelims}

We will make use of the following result of the second author. This type of theorem has a long history, going back to work of Wermer
\cite{wermer64} and \cite{wermer} and Freeman \cite{freeman} in the 1960's - for a detailed account, see \cite{izzo_approx_on_manifolds}.

\begin{theorem}[\cite{izzo_approx_on_manifolds}, Theorem~1.3] \label{approximation_theorem}
Let $A$ be a uniform algebra on a compact
 space $X$, and suppose that the maximal ideal space of $A$ is $X$.
Suppose also that $E$ is a closed subset of $X$ such that $X \setminus E$
is an $m$-dimensional manifold and such that
\begin{enumerate}
\item for each point $p \in X \setminus E$ there are functions
$f_1, \ldots, f_m$ in $A$ that are $C^1$ on $X \setminus E$ and satisfy $df_1
\wedge \ldots \wedge df_m (p) \neq 0$, and
\item  the functions in $A$ that are $C^1$ on $X \setminus E$ separate points
on $X$.
  \end{enumerate}
Then $A= \{g \in C(X) : g | E \in A|E \}$.  \end{theorem}

As the proof of Theorem \ref{support_theorem} is fairly technical we begin with some heuristics.  Suppose that in a neighborhood $N$ of a regular point $p$ of the variety $V$, we can find functions $f_{1}, \ldots, f_{m}$ in the algebra $A$ such that $F = (f_{1}, \ldots , f_{m})$ provides a smooth imbedding of $N$ into $\mathbb{C}^m$.   The condition $df_1
\wedge \ldots \wedge df_m (p) \neq 0$ appearing in Theorem \ref{approximation_theorem} implies that $F(N)$ is totally real at $F(p)$; on the other hand, if $df_1 \wedge \ldots \wedge df_m (p)$ vanishes on a relatively open subset of $F(N)$ then $F(N)$ contains a CR-manifold of positive CR dimension, such that functions in $A$ correspond to CR functions.  This would contradict the countable approximation property, which is essentially the assumption that the algebra is locally dense in the continuous functions. 

As we show below, the set on which condition (1) of Theorem \ref{approximation_theorem} fails - call it the exceptional set -  is a subvariety of the regular points of $V$; the preceding argument suggests that it has no interior relative to $V$, and so is a proper subvariety of $V$.  Theorem \ref{approximation_theorem} then implies that we can reduce approximation on $V$ to approximation on the union of the exceptional set and the singular points of $V$.  One would like to then use induction to reduce approximation (i.e., the support of a putative annihilating measure $\mu$) on $V$ to approximation on a sequence of varieties of decreasing dimension, until the dimension is zero (i.e., the variety is a discrete set), and then to conclude that the support of an annihilating measure must be empty.  It is not obvious, however, that the \emph{union} of the exceptional set and the singular set, even locally, must itself be contained in a subvariety of
$V$ of dimension less than that of $V$.  To get around this difficulty, we proceed as in \cite{anderson_izzo_smooth_manifolds}, treating the exceptional set and singular set separately, introducing a filtration of $V$ into exceptional sets and singular sets of decreasing dimensions, then first showing by induction on decreasing dimension of the exceptional sets that the support of any annihilating measure $\mu$ must lie
in the singular set.  We then use induction again on a decreasing sequence of singular sets to reduce the support of $\mu$ to the empty set.

We now prove a series of lemmas culminating in the statement that the exceptional set has empty interior (Lemma \ref{exceptional set empty interior}). The Section concludes with a review of some elementary facts about real-analytic varieties. 

\begin{lemma} \label{uniform lemma} Suppose $X$ is a compact space that is first countable and separable,
and $A$ is a uniform algebra on $X$ with the countable approximation property.  Then there exists a compact subset $N$ of $X$ with nonempty interior such that $\overline{A\big|_{N}} = C(N)$. \end{lemma}

To prove this we will need the following lemma.  As the proof is short, we include it for the reader's convenience.

\begin{lemma}[\cite{izzo_localization}, Lemma~2.4] \label{izzo lemma} Let $A$ be a uniform algebra on a compact space $X$ that is first countable and separable.  Suppose that for each function $f \in C(X)$ there is a nonempty open subset $U$ (depending on $f$) of $X$ such that $f\big|_{U} \in \overline{A\big|_{U}}$.  Then there exists a compact set $N$ with nonempty interior such that $\overline{A\big|_{N}} = C(N)$. \end{lemma}
\begin{proof} Choose a countable dense subset $\{x_{n}\}$ of $X$, and about each $x_{n}$, choose a countable basis $\{V_{m}^{n}\}_{m=1}^{\infty}$.  Then the collection $\{ V_{m}^n: m,n \in \mathbb{Z}_{+} \}$ is a countable collection of open sets such that each open subset of $X$ contains at least one of them.  For each $m$ and $n$, let $Y_{m}^n = \{ f \in C(X): f|V_{m}^{n} \in \overline{A|V_{m}^n}\}$. By hypothesis, $\cup_{m,m = 1}^{\infty} Y_{m}^n = C(X)$.  Applying the Baire category theorem shows that some $Y_{m}^{n}$ must fail to be nowhere dense in $C(X)$.  Since $Y_{m}^{n}$ is clearly a closed vector subspace of $C(X)$, we then have $Y_{m}^{n} = C(X)$.  Setting $N = \overline{V_{m}^{n}}$ gives the conclusion of the lemma.
\end{proof}

\begin{proof}
[Proof of Lemma~\ref{uniform lemma}]
Fix $f \in C(X)$.
By hypothesis, there exists a countable collection $\{M_{n}\}$ of compact subsets of $X$ with $\cup_{n=1}^{\infty} M_{n} = X$ and $f\big|_{M_{n}} \in \overline{A\big|_{M_{n}}}$ for each $n$.  By the Baire Category theorem, there exists $n$ such that $M_{n}$ has nonempty interior.  Applying Lemma~\ref{izzo lemma} with $U = \mbox{int}(M_{n})$ establishes the result. \end{proof}

The following lemma is analogous to Lemma~2.3 of \cite{anderson_izzo_smooth_manifolds}.

\begin{lemma} \label{inheritance lemma} Suppose $A$ is a uniform algebra on a compact space $X$ that is first countable and separable.
Suppose also that the maximal ideal space of $A$ is $X$ and that $A$ has the countable approximation property.  If $Y$ is a closed subset of $X$, then
\begin{description}
                                      \item[(i)] the maximal ideal space of $\overline{A\big|_{Y}}$ is $Y$, and
                                       \item[(ii)] $\overline{A\big|_{Y}}$ has the countable approximation property.
\end{description}
\end{lemma}

\begin{proof}
We prove (ii) first.
Given $f \in C(Y)$, choose $g \in C(X)$ with $g\big|_{Y} = f$.  Then there exists a countable family $\{M_{n}\}$ of compact subsets of $X$ with $\cup_{n=1}^{\infty} M_{n} = X$ and $g\big|_{M_{n}} \in \overline{A\big|_{M_{n}}}$.  The countable family  $M_{n}^{\prime} = M_{n} \cap Y$ of compact subsets of $Y$ clearly satisfies $\cup_{n=1}^{\infty} M_{n}^{\prime} = Y$ and $f\big|_{M_{n}^\prime} = g\big|_{M_{n}^\prime} \in \overline{A\big|_{M_{n}^{\prime}}}$.

We now prove (i). The maximal ideal space of $\ov{A|_Y}$ is the $A$-convex hull $\widehat Y$ of $Y$
\cite[Theorem~II.6.1]{Ga} defined by
  \[ \widehat{Y} = \{ p \in X: |f(p)| \leq \|f\|_{Y} \mbox{ for all }f \in A \}. \]
  Assume to get a contradiction that $\widehat Y$ is strictly larger than $Y$.  Let $U={\widehat Y\setminus Y}$.
By (ii), $\ov{A|_{\ov U}}$ satisfies the countable approximation property.  Hence by Lemma~\ref{uniform lemma}, there exists a compact subset $N$ of $\ov U$ with nonempty interior in $\ov U$ such that $\ov{A|_N}=C(N)$.  Note that then $N$ also has nonempty interior in $U$.  Denote this interior by $W$, and note that $W$ is an open set of $\widehat Y$ contained in $U$.    Now choose a continuous function $f:X\rightarrow [0,1]$ such that $f$ takes the value 1 at some point of $W$ and $f$ is identically zero on $X\setminus W$. Since
$\ov{A|_N}=C(N)$, there is a function $g$ in $A$ such that $\|f-g\|_N< 1/4$.
Then $\|g\|_W> 3/4$, and there is a compact set $L$ of $W$ such that $|g|<1/4$ on $W\setminus L$.  By multiplying $g$ by a suitable constant, we obtain a function $h$ in $A$ with $\|h\|_W=1$, with $h(p)=1$ for some point $p$ of $L$, and with $|h|<1/2$ on $W\setminus L$.  Then the function $(1+h)/2$ is in $A$ and, restricted to the open set $W$ of $\widehat Y$, peaks on a compact subset $E$ of $W$.  (It is necessary to take $(1+h)/2$ as the function $h$ itself may take values of modulus 1 other than the number 1.)  Now $E$ is a local peak set for $\ov{A|_Y}$, and hence, by Rossi's local peak set theorem \cite[Theorem~III.8.1]{Ga}, $E$ is a peak set for $\ov{A|_Y}$.  But this is a contradiction since $E$ is contained in $\widehat Y\setminus Y$.
\end{proof}

Given the result of Lemma~\ref{inheritance lemma}, the following will ensure that in our context the support of an annihilating measure has no isolated points.

\begin{lemma}\label{no isolated support} Let $A$ be a uniform algebra on a compact space $X$ and let $\mu$ be an annihilating measure for $A$.  Suppose that $A$ has the property that for each compact subset $Y$ of $X$ the maximal ideal space of $\overline{A\big|_{Y}}$ is $Y$.  Then the support of $\mu$ has no isolated points. \end{lemma}

\begin{proof} Let $Y$ be the support of $\mu \in A^{\perp}$.  Assume by contradiction that $Y$ has an isolated point $p$.  By the Shilov idempotent theorem \cite[Theorem~III.6.5]{Ga} applied to the algebra $\overline{A\big|_{Y}}$, there is a function $f \in \overline{A\big|_{Y}}$ such that $f(p) = 1$ and $f(Y\setminus \{p\}) = 0$.  Choose a sequence $f_{n} \in A$ with $f_{n}$ converging uniformly on $Y$ to $f$.  Then for each $n$,
\[ 0 = \int_{X} f_{n} \,d\mu = \int_{Y} f_{n} \,d\mu \]
implying that
\[ 0 = \int_{Y} f \,d\mu = \mu(\{p\}), \]
contradicting the assumption that $p$ is an isolated point in the support of $\mu$.  \end{proof}

\begin{lemma}\label{exceptional set empty interior} Let $M$ be a manifold-with-boundary of class $C^{1}$.  Suppose $K$ is a compact subset of $M$, and $A$ is a uniform algebra on $K$ generated by a collection $\Phi$ of functions that are $C^{1}$ in a neighborhood of $K$.  Suppose that $A$ has the countable approximation property.  Then the set
\[ E := \{ p \in K: df_{1} \wedge \cdots \wedge df_{m}(p) = 0 \mbox{ for all } f_{1}, \ldots ,f_{m} \in \Phi \} \]
has empty interior in $M$. \end{lemma}

\begin{proof} It suffices to show that given any open set $U$ of $M$ contained in $K$, there are functions $f_{1}, \ldots, f_{m}$ in $\Phi$ such that $df_{1} \wedge \cdots \wedge df_{m}$ is not identically zero on $U$.  We will choose these functions inductively so that for each $k$,  $1 \leq k \leq m$, $df_{1} \wedge \cdots \wedge df_{k}$ is not identically zero on $U$.   For $k = 1$, the assumption that functions in $\Phi$ separate the points of $K$ implies that there exists $f_{1} \in \Phi$ with $df_{1}$ not identically zero on $U$.    Now assume that $f_{1}, \ldots, f_{k}$ have been chosen so that $\omega := df_{1} \wedge \cdots \wedge df_{k}$ is not identically zero on $U$.  Let $W = \{ p \in U: \omega(p) \neq 0 \}$.  Set $\mathfrak{A} = \overline{A\big|_{\overline{W}}}$. By Lemma~\ref{inheritance lemma}, $\mathfrak{A}$ has the countable approximation property, so by Lemma~\ref{uniform lemma}, there exists a compact subset $N$ of $\overline{W}$ with non-empty interior in $\overline{W}$ such that $\overline{\mathfrak{A}\big|_{N}} = C(N)$; without loss of generality we may assume $N \subset W$. Fix $p \in N$.  We may choose $k$ vectors $\xi_{1}, \ldots ,\xi_{k}$ with $\omega(p)(\xi_{1}, \ldots ,\xi_{k}) \neq 0$.  We may then choose a set $B$ diffeomorphic to an open $(k+1)$-ball with $\overline{B} \subset N$ and $p \in \partial B$
such that $\xi_{1}, \ldots ,\xi_{k}$ span the tangent space to $\partial B$ at $p$.  Since $\omega$ does not vanish identically as a form on $\partial B$, there exists a smooth function $g$ on $\partial B$ with
\begin{equation} \int_{\partial B} g \omega \neq 0. \label{eq:nonzero} \end{equation}
Choose $G \in C(N)$ with $G = g$ on $\partial B$.  Then $G$ is a uniform limit on $N$ of elements of $\mathfrak{A}$, hence $g$ is a uniform limit of elements of $A$ on $\partial B$, and therefore is a uniform limit on $\partial B$ of polynomials in elements of $\Phi$.  It follows that there is a polynomial $h$ in elements of $\Phi$ such that (\ref{eq:nonzero}) holds with $g$ replaced by $h$.  By Stokes' Theorem,
\begin{equation} \int_{B} dh \wedge \omega = \int_{\partial B} h \omega \neq 0. \label{eq:nonzerotwo} \end{equation}
From the formula for the differential of a product, we see that $dh$ is a linear combination (with coefficients that are smooth functions) of differentials of functions in $\Phi$.  Thus (\ref{eq:nonzerotwo}) implies the existence of a function $f_{k+1}$ such that $\omega \wedge df_{k+1} = df_{1} \wedge \cdots \wedge df_{k+1}$ is not identically zero on $B$.  This completes the induction and the proof. \end{proof}

Although it will not be needed in the sequel, we remark that Theorem~\ref{approximation_theorem} and Lemma~\ref{exceptional set empty interior} imply a result about the essential set of certain algebras, analogous to Theorem~1.2 of \cite{anderson_izzo_smooth_manifolds}. Recall that the \emph{essential set} $E$ for a uniform algebra on a space $X$ is the smallest closed subset of $X$ such that $A$ contains every continuous function on $X$ that vanishes on $E$.  Note that if $F$ is any closed set such that
\[ A = \{g \in C(X): g\big|_{F} \in A\big|_{F} \} \]
then the essential set $E$ for $A$ is contained in $F$. Combining Theorem~\ref{approximation_theorem} and Lemma~\ref{exceptional set empty interior} we immediately obtain

\begin{theorem} Let $A$ be a uniform algebra on a compact manifold with boundary of class $C^1$.  Assume that $A$ is generated by $C^{1}$ functions, that the maximal ideal space of $A$ is $M$, and that $A$ has the countable approximation property.  Then the essential set for $A$ has empty interior in $M$. \end{theorem}

Next we recall some standard facts about varieties which will be used in the proof of Theorem~\ref{support_theorem}.

Let $\Sigma$ be a real-analytic variety in the open set $W \subset \mathbb{R}^n$, that is, for each point $p \in W$ there is a finite set $\mathcal{F}$ of functions real-analytic in a neighborhood $N$ of $p$ in $W$ with $\Sigma $ the common zero set of $\mathcal{F}$ in $N$. The following definitions and facts are standard; see for example \cite{narasimhan}. We let $\Sigma_{\mbox{\tiny reg}}$ denote the set of points $p \in \Sigma$ for which there exists a neighborhood $N$ of $p$ in $\mathbb{R}^n$ such that $\Sigma \cap N$ is a regularly
imbedded real-analytic submanifold of $N$ of some dimension $d := d(p)$.   This dimension $d(p)$ is locally constant on $\Sigma_{\mbox{\tiny reg}}$.  The dimension of $\Sigma,$ denoted by $\mbox{dim}(\Sigma)$, is defined to be the largest such $d(p)$ as $p$ ranges over the regular points of $\Sigma$.  The singular set of $\Sigma$, denoted by $\Sigma _{\mbox{\tiny sing}}$, is the complement in $\Sigma$ of $\Sigma_{\mbox{\tiny reg}}$.  If $\Sigma^{\prime} \subset \Sigma$ is a real-analytic subvariety of some open set $W^{\prime} \subset W$ and $\Sigma^{\prime}$ has empty interior relative to $\Sigma$, then $\mbox{dim}(\Sigma^{\prime}) < \mbox{dim}(\Sigma)$.  Both $\Sigma$ and $\Sigma_{\mbox{\tiny sing}}$ are closed in $W$.  Although $\Sigma_{\mbox{\tiny sing}}$ may not itself be a subvariety of $W$, it is locally contained in a proper subvariety of $\Sigma$: for each $p \in \Sigma _{\mbox{\tiny sing}}$, there is a real-analytic subvariety $Y$ of an open neighborhood $N$ of $p$ such that $\Sigma_{\mbox{\tiny sing}} \cap N \subset Y$  and $\mbox{dim}(Y) < \mbox{dim}(\Sigma)$.

If $M$ is an $m$-dimensional manifold of class $C^1$ and
$\Phi$ is a collection
of functions that are $C^{1}$ on
$M$, we define the {\em exceptional set of $M$ relative to $\Phi$}  by
\begin{equation} M_{\Phi} = \{ p \in M : df_{1} \wedge \ldots \wedge df_{m}(p) = 0 \mbox{ for all }m \mbox{-tuples } (f_{1}, \ldots , f_{m}) \in  \Phi^m \}. \label{eq:exceptional} \end{equation}
If $\Sigma$ is a real-analytic variety and $\Phi$ is a collection of functions real-analytic on $\Sigma$ (i.e., each function in $\Phi$ extends to be real-analytic in a neighborhood (depending on the function) of $\Sigma$), then $\Sigma_{\Phi}$ is defined to be the set of all points $p \in \Sigma_{\mbox{\tiny reg}}$ such that in a neighborhood of $p$, the set  $\Sigma_{\mbox{\tiny reg}}$ is an $m$-dimensional manifold $M$, and $p \in M_{\Phi}$ as defined in (\ref{eq:exceptional}).

Finally, we have:

\begin{lemma}[\cite{anderson_izzo_varieties}, Lemma~2.1] \label{exceptional_set_variety} Let $\Sigma$ be a real-analytic variety in an open set $W \subset \mathbb{R}^n$, and $\Phi$ a collection of functions real-analytic on $\Sigma$.  Then $\Sigma_{\Phi}$ is a subvariety of $W \setminus \Sigma _{\mbox{\rm\tiny sing}}$. \end{lemma}

\section{Proof of Theorems 1.3 and 1.4} \label{proof}

We first indicate how Theorem~\ref{main_theorem} can be obtained from
Theorem~\ref{support_theorem}.  Let the variety $V$, the compact set $K \subset V$, and the algebra $A$ be as in Theorem~\ref{main_theorem}.  If $\mu \in A^{\perp}$, Theorem~\ref{support_theorem} implies that $\mbox{supp}(\mu) \subset \partial K$.  Apply Theorem~\ref{support_theorem} with $V$ replaced by
$\partial K$, $K$ replaced by $\partial K$, and $A$ replaced\vadjust{\kern 2pt} by $\mathfrak{A} = \ov{A\big|_{\partial K}}$.
Lemma~\ref{inheritance lemma} implies that the maximal ideal space of $\mathfrak{A}$ is
$\partial K$ and satisfies the countable approximation property. Note that ${\rm int}(\partial K)$ relative to $\partial K$ is $\partial K$. Therefore Theorem~\ref{support_theorem} implies that $\mbox{supp}(\mu) \cap \partial K = \emptyset$, and hence $\mbox{supp}(\mu) = \emptyset$, so $\mu \equiv 0$.  This establishes Theorem~\ref{main_theorem}.

We now turn to the proof of Theorem~\ref{support_theorem}, beginning with the general construction of
\cite{anderson_izzo_varieties}.  For convenience we repeat the details.

Let $\Sigma$ be a real-analytic variety in the open set $W \subset \mathbb{R}^n$, and let $\Phi$ be a collection of functions real-analytic on $\Sigma$.  We define inductively subsets $\Sigma_{k}$ of $\Sigma$ such that $\Sigma_{0} = \Sigma$, and for $k \geq 1$, $\Sigma_{k}$ is the real-analytic subvariety of
\[ W_{k} := W \setminus \bigcup_{j=0}^{k-1} \; (\Sigma_{j})_{\mbox{\tiny sing}} \]
defined by
\[ \Sigma_{k} = (\Sigma_{k-1})_{\Phi}. \]
Note that by definition, $\Sigma_{k} \subset (\Sigma_{k-1})_{\mbox{\tiny reg}}$.
We will refer to the varieties $\Sigma_{k}$ as the {\em E-filtration of $\Sigma$ in $W$ with respect to $\Phi$}, and to the sets $(\Sigma_{k})_{\mbox{\tiny sing}}$ as the {\em S-filtration of $\Sigma$ in $W$} (\emph{E} for exceptional, \emph{S} for singular).

\begin{lemma} \label{dimension_lemma} With $V,\Omega,K, A, \Phi$ as in Theorem~\ref{support_theorem}, suppose that $W$ is an open subset of $\Omega$  and $\Sigma \subset {\rm int}(K) \cap W$ is a real-analytic subvariety of $W$.  Let $\{ \Sigma_{k} \}$ be the E-filtration of $\Sigma$ in $W$ with respect to $\Phi$.  Then for each $k$, the dimension of $\Sigma_{k}$ is no more than $\mbox{dim} (\Sigma) - k$. \end{lemma}

\begin{proof} Let $d = \mbox{dim}(\Sigma)$. The proof is by induction on $k$. The result is clear when $k = 0$.  Suppose we have shown for some $k$ that $\mbox{dim}(\Sigma_{k})\leq d - k$.  Fix $p \in (\Sigma_{k})_{\mbox{\tiny reg}}$, and let $U$ be a smoothly bounded
neighborhood of $p$ in $(\Sigma_{k})_{\mbox{\tiny reg}}$ with $\overline{U} \subset (\Sigma_{k})_{\mbox{\tiny reg}}$.  We may assume that
$U$ has constant dimension (by induction, no more than $d - k$) as a submanifold of
$\mathbb{R}^n$.  Lemma~\ref{inheritance lemma} implies that we may apply
Lemma~\ref{exceptional set empty interior} taking $M =\overline{U}$ and replacing $A$ with $\overline{A\big|_{\overline{U}}}$.  The conclusion implies that $\Sigma_{k+1} = (\Sigma_{k})_{\Phi}$ has no interior in $U$.  Since $p$ was arbitrary, we conclude that
\[ \mbox{dim}(\Sigma_{k+1}) \leq \mbox{dim}(\Sigma_{k}) - 1 \leq d - (k+1). \]
By induction, the proof is complete. \end{proof}

Note that Lemma~\ref{dimension_lemma} implies, with $d = \mbox{dim}(\Sigma)$,  that $\Sigma_{d}$ is a zero-dimensional variety, i.e., is a discrete set and hence, in particular, is at most countable.

Let $B(p,r)$ denote the open ball of radius $r$ centered at $p \in \mathbb{R}^n$.

\begin{lemma} \label{support_in_S_filtration} With $V,\Omega,K, A, \Phi$ as in Theorem~\ref{support_theorem}, assume $p \in {\rm int}(K)$ and $\mu\in A^\perp$.  If $r > 0$ is such that $B(p,r) \cap V \subset {\rm int}(K)$,
and there is a real-analytic
$d$-dimensional subvariety $\Sigma \subset V$ of $B(p,r)$ with\/ $\mbox{\rm supp}(\mu) \cap B(p,r) \subset \Sigma$, then\/ $\mbox{\rm supp}(\mu) \cap B(p,r)$ is contained in the S--filtration of $\Sigma$, i.e., $\mbox{\rm supp}(\mu) \cap B(p,r)\subset
\bigcup\limits_{k=0}^{d-1} (\Sigma_{k})_{\mbox{\rm\tiny sing}}$.
\end{lemma}

\begin{proof}
We will show by induction on $L$ that
\[ \mbox{supp}(\mu) \cap B(p,r) \subset \left (\,\bigcup_{k=0}^{L} (\Sigma_{k})_{\mbox{\tiny sing}} \right ) \cup \Sigma_{L+1} \]
for each $L$, $0 \leq L \leq d-1$.
This suffices: since $\Sigma_{d}$ is discrete and disjoint from
the relatively closed subset $\left (\,\bigcup_{k=0}^{d-1} (\Sigma_{k})_{\mbox{\tiny sing}} \right )$ of $B(p,r)$, each point of $\Sigma_{d}$ is an isolated point of the set  $\left (\,\bigcup_{k=0}^{d-1} (\Sigma_{k})_{\mbox{\tiny sing}} \right ) \cup \Sigma_{d}$, and so
Lemmas~\ref{inheritance lemma} and~\ref{no isolated support} imply that $\mu(\Sigma_d)=0$.

For the $L = 0$ case, let $X = (K \setminus B(p,r)) \cup \Sigma_{0} $ and let $E =(K \setminus B(p,r)) \cup (\Sigma_{0})_{\mbox{\tiny sing}} \cup \Sigma_{1}$.  Note that both $X$ and $E$ are closed.  We want to show that $\mbox{supp}(\mu)\subset E$.   By Lemma~\ref{inheritance lemma}, the maximal ideal space of
$\overline{A|X}$ is $X$.  Note that $X \setminus E = \Sigma_{0} \setminus \bigl((\Sigma_{0})_{\mbox{\tiny sing}} \cup \Sigma_{1}\bigr)$
satisfies the hypotheses of Theorem~\ref{approximation_theorem}.
  Therefore by Theorem~\ref{approximation_theorem}, if $g \in C(K)$ vanishes on $E$, then $g|X$ belongs to
  $\ov{A|X}$.  Since by hypothesis $\mbox{supp}(\mu) \subset X$, we get that $\int_{K} g \; d\mu = 0$ for each $g \in C(K)$ vanishing on $E$, and this implies that $\mbox{supp}(\mu) \subset E$, as desired.

The general induction step is similar: assuming the result for some $0\leq L < d-1$, we set
\[ X = (K \setminus B(p,r)) \cup \left ( \bigcup_{k=0}^{L} (\Sigma_{k})_{\mbox{\tiny sing}} \right ) \cup \Sigma_{L+1}, \]
\[ E = (K \setminus B(p,r)) \cup \left ( \bigcup_{k=0}^{L+1} (\Sigma_{k})_{\mbox{\tiny sing}} \right ) \cup \Sigma_{L+2}. \]
Both $X$ and $E$ are closed.  Noting that the induction hypothesis implies that $\mbox{supp}(\mu) \subset X$, we apply Theorem~\ref{approximation_theorem} to $\ov{A|X}$ as above, to conclude that $\mbox{supp}(\mu) \subset E$.  \end{proof}

\begin{lemma} \label{support_in_boundary} With $V,\Omega,K, A, \Phi$ as in Theorem~\ref{support_theorem}, assume $p \in {\rm int}(K)$
and $\mu\in A^\perp$.  Assume also that there exist $r > 0$ with $B(p,r) \cap V \subset {\rm int}(K)$
and a real-analytic subvariety $\Sigma \subset V$ of $B(p,r)$ with\/ $\mbox{\rm supp}(\mu) \cap B(p,r)$ contained in the S--filtration of $\Sigma$ in $B(p,r)$.  Then there exists $r^{\prime} > 0$ such that\/ $\mbox{\rm supp}(\mu) \cap B(p,r^{\prime}) = \emptyset$. \end{lemma}

\begin{proof} We apply induction on the dimension of $\Sigma$.  If $\mbox{dim}(\Sigma) = 0$, then $\Sigma$ is discrete, and by Lemma~\ref{no isolated support}
this implies $|\mu|(B(p,r)) = |\mu|(\Sigma) = 0$.  Now suppose the conclusion of the Lemma~holds whenever $\mbox{dim}(\Sigma) < d$.  If $\mbox{dim}(\Sigma) = d$, let $\Sigma_{0}, \ldots, \Sigma_{d}$ be the \emph{E}-filtration of $\Sigma$ in $B(p,r)$.  (Recall that $\Sigma_{d}$ is discrete.) By induction on $L$ we will show that
\begin{equation} \mbox{supp}(\mu) \cap B(p,r) \subset \bigcup_{k=0}^{d-1-L} (\Sigma_{k})_{\mbox{\tiny sing}} \label{eq:supp} \end{equation}
for $L = 0, \ldots , d-1$.  The case $L = 0$ is the hypothesis of the Lemma.  Assume we have established (\ref{eq:supp}) for some $L$, $0 \leq L < d-1$.
To show that (\ref{eq:supp}) holds with $L$ replaced by $L + 1$, we must show that $\mbox{supp}(\mu) \cap (\Sigma_{d-1-L})_{\mbox{\tiny sing}} = \emptyset$.  Fix $q \in (\Sigma_{d-1-L})_{\mbox{\tiny sing}}$.  By construction of the $\Sigma_{k}$, there exists $s > 0$ so that $B(q,s) \subset B(p,r)$ and $B(q,s) \cap (\Sigma_{k})_{\mbox{\tiny sing}} = \emptyset$ for all $k < d-1-L$. Therefore the induction hypothesis implies that $\mbox{supp}(\mu) \cap B(q,s) \subset (\Sigma_{d-1-L})_{\mbox{\tiny sing}}$.  Replacing $s$ by a smaller positive number if necessary, we may assume that there is a real-analytic subvariety $Y$ of $B(q,s)$ with $(\Sigma_{d-1-L})_{\mbox{\tiny sing}} \subset Y \subset V$ and $\mbox{dim}(Y) < \mbox{dim}(\Sigma_{d-1-L}) \leq L + 1 \leq d$ (the next-to-last inequality following from Lemma~\ref{dimension_lemma}).  By Lemma~\ref{support_in_S_filtration}, $\mbox{supp}(\mu) \cap B(q,s)$ is contained in the \emph{S}-filtration of $Y$ in $B(q,s)$.  Now note that our induction hypothesis on dimension implies that the conclusion of Lemma~\ref{support_in_boundary} holds with $\Sigma$ replaced by $Y$, since $\mbox{dim}(Y) < d$.  We conclude that there exists $s'>0$ such that $\mbox{supp}(\mu) \cap B(q,s') = \emptyset$.  Since $q  \in (\Sigma_{d-1-L})_{\mbox{\tiny sing}}$ was arbitrary, this shows that $\mbox{supp}(\mu) \cap (\Sigma_{d-1-L})_{\mbox{\tiny sing}} = \emptyset$ and completes the proof that (\ref{eq:supp}) holds for $L = 0, \ldots , d-1$.

Finally, the case $L = d-1$ of (\ref{eq:supp}) asserts that $\mbox{supp}(\mu) \cap B(p,r) \subset (\Sigma_{0})_{\mbox{\tiny sing}}$.   We may choose $t$ with $0 < t< r$ and a subvariety $Y \subset V$ of $B(p,t)$ so that
$(\Sigma_{0})_{\mbox{\tiny sing}}\cap B(p,t) \subset Y$ and $\mbox{dim}(Y) < \mbox{dim}(\Sigma) = d$.
By Lemma~\ref{support_in_S_filtration}, $\mbox{supp}(\mu) \cap B(p,t)$ is contained in the \emph{S}-filtration of $Y$ in $B(p,t)$.  Again applying our induction hypothesis on dimension, we conclude that there exists $r'>0$ such that $\mbox{supp}(\mu) \cap B(p,r') = \emptyset$.  This completes the proof.  \end{proof}

We can now finish the proof of Theorem~\ref{support_theorem}.  Let $\mu \in A^\perp$ be given.  Fix $p \in {\rm int}(K)$, and $r > 0$ such that $B(p,r) \cap V \subset \mbox{int}(K)$.  Taking $\Sigma = V \cap B(p,r)$ in Lemma~\ref{support_in_S_filtration}, we see that $\mbox{supp}(\mu) \cap B(p,r)$ is contained in the \emph{S}--filtration of $V \cap B(p,r)$.  By Lemma~\ref{support_in_boundary}, there is an $r'>0$ such that $\mu$ has no support in $B(p,r')$, concluding the proof.
\bigskip\eject

\end{document}